\documentclass[12pt, USletter, reqno, oneside, makeidx]{amsart}
\usepackage[margin=2.7cm, marginpar=1cm]{geometry}

\usepackage{verbatim}
\usepackage{amsmath}
\usepackage{amssymb}
\usepackage{graphicx}
\usepackage{color}
\usepackage{empheq}
\usepackage{tikz-cd}
\usepackage[autostyle]{csquotes}
\usepackage[makeroom]{cancel}
\usepackage{enumitem}
\usepackage{hyperref}

\usepackage{tikz}
\usetikzlibrary{calc}

\theoremstyle{plain}
\newtheorem{thm}{Theorem}[section]
\newtheorem{lem}[thm]{Lemma}
\newtheorem{proposition}[thm]{Proposition}
\newtheorem{defi}[thm]{Definition}
\newtheorem{cor}[thm]{Corollary}
\newtheorem{conj}[thm]{Conjecture}

\newtheorem{quest}[thm]{Question}
\newtheorem{remark}[thm]{Remark}
\newtheorem{exa}[thm]{Example}

\newcommand{\Z}{\mathbb{Z}}
\newcommand{\Q}{\mathbb{Q}}
\newcommand{\R}{\mathbb{R}}
\newcommand{\F}{\mathbb{F}}

\newcommand{\CF}{\text{CF}}

\newcommand{\Hy}{\mathbb{H}}

\newcommand{\bs}{\mathbf{s}}

\newcommand{\x}{\mathbf{x}}

\newcommand{\y}{\mathbf{y}}

\newcommand{\s}{\mathfrak{s}}

\newcommand{\Spinc}{\text{Spin}^c}

\DeclareMathOperator{\HFL}{HFL}

\DeclareMathOperator{\HF}{HF}

\DeclareMathOperator{\CFL}{CFL}
\DeclareMathOperator{\CFK}{CFK}

\newcommand{\spin}{\text{spin}}

\begin{document}
\title[]{Is the geography of Heegaard Floer homology restricted or the L-space conjecture false?}
\author{ Antonio Alfieri and Fraser Binns }
\address{Centre de recherche Math\'ematique. Montreal, Canada.}
\email{alfieriantonio90@gmail.com}
\address{Princeton University. Princeton, New Jersey, USA.}
\email{fb1673@princeton.edu}
\thanks{FB was supported by the Simons Grant {\em New structures in low-dimensional topology}.}

\maketitle

\begin{abstract}
In a recent note F. Lin showed that if a rational homology sphere $Y$ admits a taut foliation then the Heegaard Floer module $\HF^-(Y)$ contains a copy of $\F[U]/U$ as a summand~
\cite{lin2023remark}. This implies that either the $L$-space conjecture is false or that Heegaard Floer homology satisfies a geography restriction. We verify that Lin's geography restriction holds for a wide class of rational homology spheres. Indeed, we show that the Heegaard Floer module $\HF^-(Y)$ may satisfy a stronger geography restriction.

\end{abstract}

\thispagestyle{empty}

\section{Introduction}
Heegaard Floer homology is a package of invariants introduced by Ozsv\'ath and Szab\'o in 2001~\cite{ozsvath2004holomorphic}. It assigns to each three-manifold $Y$ an $\F[U]$-module, $\HF^-(Y)$, that can be used to study -- for example -- surgery problems, contact structures on three-manifolds, and the symplectic topology of four-manifolds; see~\cite{lisca2004ozsvath,lisca2007ozsvath,lisca2007ozsvath3,ozsvath2004holomorphictrinaglesymplectic,ni2010cosmetic}.
 
While a quarter of a century has passed since it was defined, there are still many mysterious open questions regarding Heegaard Floer homology. One of these asks if the structure of the module $\HF^-(Y)$ is related to the existence of taut foliations on $Y$.

Recall that a foliation of a three-manifold $Y$ is a decomposition of $Y$ into a disjoint union of $2$-dimensional sub-manifolds (leaves) that locally looks like the decomposition of $\R^3$ into horizontal planes. Every three-manifold $Y$ admits a foliation. However, it is not necessarily the case that every manifold admits a foliation if additional constrains are required -- for example if the foliation is asked to be ``Reebless" or ``taut". In the other direction, if a manifold admits such a foliation, this tells us something about the topology of $Y$ \cite{novikov1965topology,palmeira1978open}.

 A foliation is \emph{taut} if there exists a simple closed curve that intersects every leaf. In the 1980's Gabai \cite{gabai1983foliations,gabai1987foliations,gabai1987foliations3} showed that every three-manifold with $b_1>0$ admits a taut foliation. At the time of writing it is unknown which \emph{rational homology spheres} (connected, oriented, three manifolds with $b_1=0$) admit taut foliations. Understanding this problem is one of the greatest open challenges of three-manifold topology.

 There is a conjectural relationship between rational homology spheres that admit taut foliations and those whose Heegaard Floer homology is free as follows:

\begin{conj}[Boyer, Gordon and Watson~\cite{boyer2013spaces}, Juh\'asz~\cite{juhasz2015survey}: the $L$-space conjecture] For an irreducible rational homology sphere $Y$ the following statements are equivalent: 
\begin{enumerate}
\item $Y$ is an $L$-space -- i.e. $\HF^-(Y)$ is free.
\item $Y$ does not support a taut foliation,
\item $\pi_1(Y)$ has no left-ordering; that is an ordering $<$ such that: $a<b$ iff $ga<gb$.  
\end{enumerate}
\end{conj}

Note that the equivalence of (1) and (3) would pin down a property of the fundamental group characterizing $L$-spaces, thus answering one of the fundamental questions in Ozsv\' ath and Szab\' o's problem list \cite[Question 11]{ozsvath2004heegaardsurvey}.

\begin{quest} Is there a topological classification of $L$-spaces?
\end{quest}

In this paper we advance the hypothesis that the $L$-space conjecture is linked to another fundamental problem in Heegaard Floer theory:

\begin{quest}[The Geography Question]
    Which $\F[U]$-modules are realized as the Heegaard Floer group $\HF^-(Y)$ of some three-manifold $Y$?
\end{quest}

It is immediate from the definition that $\HF^-(Y)$ is finitely generated for any three-manifold manifold $Y$. Indeed, Ozsv\'ath and Szab\'o showed that for a rational homology sphere $Y$, $\HF^-(Y)$ is of rank $|H_1(Y;\Z)|$ as a $U$-module~\cite[Theorem 10.3]{ozsvath2004holomorphic}. 

\subsection{Lin's theorem about taut foliations} The $L$-space conjecture holds for graph manifolds, some families of knot surgeries, and some branched coverings -- see~\cite{boyer2023recalibrating,hanselman2020spaces,boyer2017foliations,li2014taut}, for example. Nevertheless there is no consensus on whether the $L$-space holds true in full generality. The only part of the conjecture which is known to be true in general is the following: if a rational homology sphere $Y$ admits a taut foliation then the Heegaard Floer group $\HF^-(Y)$ has non-trivial torsion~\cite[Theorem 1.4]{ozsvath2004holomorphicdisksandgenusbounds}. Recently Lin proved the following refinement of this result.

\begin{thm}[Lin~\cite{lin2023remark}]\label{thm:lin}
Suppose $Y$ admits a taut foliation. Then the Heegaard Floer module $\HF^-(Y)$ contains a copy of $\F[U]/U$ as an $\F[U]$-summand.
\end{thm}

For a rational homology sphere  $Y$, we say that $\HF^-(Y)$ satisfies the \emph{Lin geography restriction} if $\HF^-(Y)$ has either no $U$-torsion or an $\F[U]/U\simeq \F$ summand. In what follows we use Heegaard Floer surgery formulae \cite{manolescu2010heegaard,ozsvath2010knotrationalsurgeries,Rasmussen} to prove that a broad class of rational homology spheres have Heegaard Floer homology satisfying the Lin geography restriction.  Thus we present evidence for the following conjecture.

\begin{conj}\label{con:linrestriction}The Heegaard Floer module $\HF^-(Y)$ satisfies the Lin geography restriction for all rational homology spheres $Y$.
\end{conj}

It follows from Lin's theorem, Theorem~\ref{thm:lin}, that either Conjecture \ref{con:linrestriction} is true, or the $L$-space conjecture is false. Note that for a large class of graph manifolds the Lin geography restriction was first verified by Bodnar, and Plamenevskaya \cite{bodnar2021heegaard}. 

\subsection{Statement of the main results} We shall make use of the following definition. 

\begin{defi}
An $\F[U]$-module $M$ satisfies the strong geography restriction if it contains a direct summand
\[\F[U]/U^\ell \oplus \F[U]/U^{\ell-1}\oplus \dots \oplus  \F[U]/U^2 \oplus \F[U]/U\]
where $\ell=\min \{\ell\geq0 : U^\ell \cdot M_\text{red}=0\}$. 
\end{defi}

Our first theorem concerns surgeries on knots. 

\begin{thm} 
\label{knots} If $Y$ is surgery on a knot in $S^3$ then the Heegaard Floer module $\HF^-(Y)$ satisfies the strong geography restriction.
\end{thm}

 In particular if $Y$ is surgery on a knot then either $Y$ is an $L$-space, or $\HF^-(Y)$ has an $\F[U]/U$-summand, \text{i.e.} it satisfies the Lin geography restriction. This answers a question that was recently asked by Baldwin on his research blog \cite{BaldwinBlog}. Our second theorem covers large surgeries on links.

\begin{thm} \label{largelinks}  If $Y$ is large surgery on a link in $S^3$ then the Heegaard Floer module $\HF^-(Y)$ satisfies the strong geography condition.  
\end{thm}

Here ``large surgery on $L$" means that the framing matrix of $L$ is ``larger" than the smallest hyperbox containing the link Floer homology polytope of $L$ -- see Definition~\ref{def:large} for further details. Recall that the link Floer homology polytope of $L$ contains equivalent information to the Thurston norm of $L$, under mild hypothesese~\cite[Theorem 1.1]{ozsvath2008linkFloerThurstonnorm}. In the case of knots ``large" simply means at least $2g(K)-1$.

\begin{remark}
Theorem \ref{knots} and Theorem \ref{largelinks} hold true in the more general case of null-homologous knots and links in any $L$-space. 
\end{remark}

Given the above results it is natural to ask the following question.

\begin{quest} \label{question1} Does $\HF^-(Y)$ satisfy the strong geography restriction for every rational homology sphere $Y$?
\end{quest}

Of course, given that every rational homology sphere is integral surgery on \emph{some} link, removing the hypothesis that $Y$ is \emph{large} surgery on a link in the statement of Theorem \ref{largelinks} would provide an affirmative answer to Question~\ref{question1}. Indeed, we know many examples of links -- for instance torus links, and connected sums of Hopf links -- whose small surgeries satisfy at least the Lin geography restriction.

\subsection{The realization problem} Theorem \ref{largelinks} leads naturally to the following question.

\begin{quest}\label{surgeryquest}
Which rational homology spheres arise as large surgery along links in $S^3$? 
\end{quest}

There are relatively few results in this direction. Gordon-Luecke showed that a reducible integral homology sphere can never be surgery on a knot~\cite{gordon1987only}. Examples of irreducible rational homology spheres which are not surgery on a knot in $S^3$ have been given by Hom-Lidman-Karakurt~\cite[Theorem 1.1]{hom2016surgery} , Hom-lidman~\cite[Theorem 1]{hom2018note}, and Auckley~\cite{auckly1997surgery}. Other examples have been given by Doig~\cite[Corollary 5]{doig2015finite}, Hoffman-Walsh ~\cite[Theorem 4.4]{hoffman2015big} and Oszv\'ath-Szab\'o~\cite[Section 10.2]{ozsvath2003absolutely}.

\begin{proposition}\label{cor:Fsummands}
        Suppose an integral homology sphere $Y$ is large surgery on a link. Then $\HF^-(Y)\cong \F[U]\oplus\F^{\oplus n}$ for some $n$.
\end{proposition}

For the Brieskorn homology spheres $\Sigma(2,2n+1,4n+3)$ Rustamov~\cite{rustamov2003calculation} showed that $\HF^-(\Sigma(2,2n+1,4n+3))$ has an $\F[U]/U^2$ summand for $n\geq 3$.

\begin{cor}
  For $n\geq 3$, $\Sigma(2,2n+1,4n+3)$ is not large surgery on a link in $S^3$. 
\end{cor}

\begin{subsection}*{Acknowledgments}
    The authors would like express their gratitude to Subhankar Dey, Franceso Lin, Duncan McCoy, Peter Ozsv\'ath, Jacob Rasmussen, Charles Stine, Andr\'as Stipsicz, and Zoltan Szab\' o  for a number of informative conversations.
\end{subsection}

\section{Three-manifolds arising as surgery on knots}
\subsection{Heegaard Floer groups} Recall that Heegaard Floer homology \cite{ozsvath2004holomorphic} associates to a pointed Heegaard diagram $\mathcal{H}=(\Sigma, \alpha_1\dots \alpha_g , \beta_1, \dots , \beta_g, z)$ of a three-manifold $Y$ a $\Q$-graded, finitely generated, free chain complex $\CF^-(\mathcal{H})$ over $\F[U]$. The differential counts $J$-holomorphic disks in the symmetric product $\text{Sym}^g(\Sigma)$. 

Since the chain homotopy type of $\CF^-(\mathcal{H})$ does not depend on $\mathcal{H}$ but only on the three-manifold $Y$ it is denoted by $\CF^-(Y)$. We shall also consider the localization $\CF^\infty(Y)= \CF^-(Y)\otimes \F[U, U^{-1}]$ and the quotient $\CF^+(Y)=\CF^\infty(Y)/\CF^-(Y)$. Note that the homologies of these chain complexes are related by an exact sequence
 \[\dots \to \HF^-(Y)\to \HF^\infty(Y) \to \HF^+(Y)\to \dots  \ ,\]
and that they decompose as direct sums along $\spin^c$ structures: 	
\[\CF^\circ(Y)=\bigoplus_{\s\in \Spinc(Y)} \CF^\circ(Y, \s)\ .\] 
This gives rise to a splitting $\HF^\circ(Y)=\bigoplus_{\s\in \Spinc(Y)} \HF^\circ(Y, \s)$ in homology, where  $\circ$ denotes any of the symbols $-,\infty, +$.

\begin{remark}The group $\HF^-(Y)$ is a finitely generated, graded $\F[U]$-module. Since $\F[U]$ is a principle ideal domain, $\HF^-(Y)$ can be decomposed as a direct sum of elementary modules of the form: 
\begin{enumerate}
    \item $T^-=\F[U]$ so called towers, and 
    \item $T^-_k:=F[U]/U^k$, $k\in\Z^{\geq 0}$.
\end{enumerate}
Note that for a rational homology sphere there is exactly one tower for each $\spin^c$ structure, that is $\text{rk}_{\F[U]} \  \HF^-(Y) =|\text{Spin}^c(Y)|=|H_1(Y;\Z)|$~\cite[Theorem 10.3]{ozsvath2004holomorphic}.
\end{remark}

\begin{remark}The group $\HF^+(Y)$ is complete torsion when considered as a module over $\F[U]$. That is: $U^i\cdot \alpha=0$ for all $\alpha \in \HF^+(Y)$, and $i>>0$. Nevertheless, there is a decomposition $\HF^+(Y)= F\oplus \HF^+_\text{red}(Y)$ where $F$ is a direct sum of a few copies of $T^+=\F[U,U^{-1}]/U\cdot \F[U]$ , and $\HF_\text{red}^+(Y) \simeq \text{Tor}(\HF^-(Y))$. When working with the $\HF^+$ version we prefer the notation $T^+_k$. 
\end{remark}
 
\subsection{Knot Floer groups: one complex to rule them all} Heegaard Floer homology also associates to each knot $K\subset S^3$ a filtered chain complex that can be used to compute the Heegaard Floer homology groups of all surgeries along $K$. 
  
\begin{defi}A  knot type chain complex is a finitely-generated, graded, $\Z \oplus\Z$-filtered chain complex 
\[C_*= \left( \bigoplus_{\x \in B} \F[U, U^{-1}],\  \partial\  \right)\] 
over the ring $\F[U,U^{-1}]$ satisfying the following properties
\begin{itemize}
\item $\partial$ is $\F[U, U^{-1}]$-linear, and given a basis element $\x \in B$, $\partial \x = \sum_\y n_{\x, \y}U^{m_{\x,\y}} \cdot \y$ for suitable coefficients $ n_{\x, \y} \in \F$, and non-negative exponents $m_{\x, \y} \geq 0$,
\item multiplication by $U$ drops the absolute $\Q$-grading  (Maslov grading) $M$ by two, and the two $\Z$-filtrations (denoted by $\mathcal{F}$ and $j$, and called the Alexander and algebraic filtrations) by one.
\item there is a grading-preserving  isomorphism $H_*(C_*, \partial)\simeq  \F[U, U^{-1}]$. 
\end{itemize} 
\end{defi}

A knot $K\subset S^3$ can be encoded by the choice of an additional basepoint $w$ in a suitable Heegaard diagram  $\mathcal{H}$ of $S^3$. As explained in \cite{ozsvath2005knot} and~\cite{Rasmussen}, this gives rise to a filtration $\mathcal{F}_K$ of $\CF^\infty(\mathcal{H})$ and hence to a knot type complex 
\[\CFK^\infty(K)=(\CF^\infty(S^3), \mathcal{F}_K, j) \ , \] 
where $j$ denotes the filtration with  $j(\x)=0$ for every intersection point $\x$ of the Heegaard diagram.
A knot type chain complex is usually presented by a planar drawing. This is done according to the following procedure:
\begin{itemize}
\item represent each $\F$-generator $U^m \cdot \x$ of $C_*$, $\x \in B$, as a point on the planar lattice $\Z \times \Z
  \subset \R^2$ in position $(i,j)=\left(A(\x)-m, j(\x)-m \right) \in \Z \times
  \Z$,
\item label each $\Z_2$-generator $U^m \cdot \x$ of $C_*$ with its Maslov grading $M(\x)-2m\in \Z$,
\item connect two $\F$-generators $U^n \cdot \x$ and $U^m \cdot \y $ with an arrow if in the differential of $U^n \cdot \x$ the coefficient of $U^m \cdot \y$ is non-zero.
\end{itemize}
With these conventions in place any region of the plane $S\subset \R^2$ identifies a sub-vector space $C_*S=\text{Span}_\F\{U^m \cdot \x\in C_*: \left(A(\x)-m , j(\x)-m\right) \in S\}$ of $C_*$. Note that for $C_*S$ to be a sub-complex of $C_*$ the regions $S$ needs to satisfy a certain convexity condition:

\begin{defi}
A region of the Euclidean plane $S \subset \R^2$ is south-west if it is closed, non-empty, and such that: $(\overline{x}, \overline{y}) \in S \Rightarrow \{ (x,y) \ | \ x\leq \overline{x}, y \leq \overline{y}\} \subseteq S.$
\end{defi}

The other case in which $C_*S$ is a chain complex is when $S=S_1\setminus S_2$ can be expressed as the difference of two south-west regions $S_2\subset S_1$, and $C_*S$ is identified with  quotient chain complex $C_*S_1/C_*S_2$. Note that the latter is not free: if $x\in S_1$ is a generator and $U^i\cdot x \in S_2$ then $U^i\cdot x=0$ in $C_*S$. 

Thus  we get chain complexes $C_*^-=C_*\{j\leq 0\}$, and $C_*^+=\{j\geq 0 \}=C_*/C_*^-$. In the case of knot Floer homology these chain complexes compute $\HF^-(Y)=\F[U]$ and $\HF^+(Y)=\F[U,U^{-1}]/U\F[U]$ respectively.

\subsection{Review of the knot surgery formula} \label{sec:surgery}We now review the homological algebra governing the knot surgery formula. 

Let $C_*$ be a knot type chain complex. For $s\in \Z$ let $A_s^+=C_*\{\max(i, j-s)\geq 0 \}$ and $B^+_s=C_*\{j\geq 0\}=B^+$. We define maps $v_s^+:A_s^+\to B^+$ and $h_s^+: A_s^+\to B^+$. The first is the quotient 
\[v_s^+:A_s^+ =C_*\{\max(i, j-s)\geq 0 \} \to C_*\{i\geq 0\}=B^+  \] 
with kernel $C_*\{j\geq s\text{ and }i<0\}$. The definition of $h_s^+$ is slightly more complicated and requires the choice of some additional data as in the following definition.

\begin{defi}[Flip map] A flip map is a chain homotopy equivalence $n: C_*\{j\geq 0\} \to C_*\{i\geq 0\}$.  A master type complex is a pair $(C_*, n)$ where $C_*$ is a knot type complex, and $n$ is a flip map of $C_*$.  
\end{defi} 

Note that if $C_*=\CFK^\infty(K)$ is the knot Floer complex of some knot in $S^3$, $C_*\{i\geq 0\}$ and  $ C_*\{j\geq 0\}$ represent two chain complexes for $\CF^+(S^3)$ with respect to the choice of two different basepoints on the same Heegaard diagram. The flip map $n: C_*\{j\geq 0\} \to C_*\{i\geq 0\}$ is induced by a sequence of Heegaard moves taking one basepoint into the other. We call $(\CFK^\infty(K), n)$ the \emph{knot Floer master complex}.

For a master type complex $(C_*,n)$ we define $h_s^+: A_s^+\to B^+$ as the composition 
\[h^+_s: A_s^+=C_*\{\max(i, j-s)\geq 0 \} \to C_*\{j\geq s\}\to C_*\{j\geq 0\} \to C_*\{i\geq 0\} =B^+ \]
where the first map is the projection $\pi:C_*\{\max(i, j-s)\geq 0 \} \to C_*\{j\geq s\}$,  the second map $C_*\{j\geq s\}\to C_*\{j\geq 0\}$  is multiplication by $U^s$, and the last one is the flip map $n: C_*\{j\geq s\} \to C_*\{i\geq 0\}$. 

Finally we set
\[\mathbb{A}^+=\bigoplus_{s\in \Z} A^+_s \ \ \ \text{ and }\ \ \ \mathbb{B}^+= \bigoplus_{s\in \Z} B^+_s \ .\] 
We have maps $v^+=\sum_s v_s^+: \mathbb{A}^+\to \mathbb{B}^+$,  and $h^+=\sum_s h_s^+: \mathbb{A}^+\to \mathbb{B}^+$. We define the map $D_p^+ :\mathbb{A}^+\to \mathbb{B}^+$ as  the sum $D_p^+: v^++ \sigma_p 	\circ h^+ $, where $\sigma_p: \bigoplus_{s\in \Z} B^+\to \bigoplus_{s\in \Z} B^+$ denotes the shift map $\sigma_p:(\dots, a_s, \dots ) \mapsto (\dots, b_s=a_{s-p}, \dots )$. Thus:  
\[D^+_p: (\dots, a_s, \dots) \mapsto(\dots,b_s= v_s^+(a_s)+h_{s-p}^+(a_{s-p}), \dots) \ .\]
Taking the mapping cone of the map $D_p^+ :\mathbb{A}^+\to \mathbb{B}^+$  we get a chain complex $\mathbb{X}^+(p) = \left( \mathbb{A}^+ \oplus \mathbb{B}^+ \ , \partial_\text{cone} = \partial + D_p^+\right)$
 that in \cite{ozsvath2004knotintegersurgery} Ozsv\'ath and Szab\' o show to be quasi-isomorphic to $\CF^-(S^3_{p}(K))$.

There is also a rational surgery formula.  Set $\mathbb{A}_i^+=\underset{s\in\Z}{\bigoplus}(s,A^+_{\lfloor\frac{i+ps}{q}\rfloor}(K))$ and ${\mathbb{B}^+=\underset{s\in\Z}{\bigoplus}(s,B^+)}$. Let $\mathbb{X}^+_{i,p/q}$ be the chain complex defined as the mapping cone of the map $D^+_{i,p/q,j}:\mathbb{A}_i^+\to \mathbb{B}^+$, which we define by $(s,A_{\lfloor\frac{i+ps}{q}\rfloor})\to (s,B^+)+ (s+1,B^+)$ by $\psi_m:(s,x)\mapsto ((s,v_m^+(x)),(s+1,h_m^+(x)),)$ and extending linearly. In~\cite{ozsvath2010knotrationalsurgeries} Ozsv\' ath and Szab\' o show that $\HF^+(S^3_{p/q}(K))$ is isomorphic to $H_*(\mathbb{X}^+_{i,p/q})$.

\subsection{Modules with similar skylines}

The following definition and lemma will be used repeatedly in the remainder of the paper.

\begin{defi} Let $A$ be a finitely generated $\F[U]$ module with 
\[A\simeq \underset{1\leq i\leq n}{\bigoplus}T^-\oplus\underset{1\leq j\leq m}{\bigoplus}T_{h(j)}^- \ .\] 
where $h(j)>0$ for all $j$. If $B$ is a finitely generated $\F[U]$ module we say that $A$ and $B$ have similar skylines if  
\[B\simeq \underset{1\leq i\leq n}{\bigoplus}T^-\oplus\underset{1\leq j\leq m}{\bigoplus}T_{h'(j)}^-\oplus\F^k\] 
for some function $h'$ with $|h'(j)-h(j)|\leq 1$ for all $j$ and $k\geq 0$.
\end{defi}

\begin{remark}
 Note that ``having similar skylines" is a reflexive relation: if $B$ has a similar skyline to $A$ then $A$ has a similar skyline to $B$. 
\end{remark}

  We say that an $\F[U]$-module $M$ is \emph{of $\HF^+$-type} if $$M\simeq \underset{1\leq i\leq n}{\bigoplus}T^+\oplus\underset{1\leq j\leq m}{\bigoplus}T_{h(j)}^+,$$ for some $m,n,h$. In this case we write \[M_\text{red}=\underset{1\leq j\leq m}{\bigoplus}T_{h(j)}^+\] for the so called \emph{reduced} part.   Note that for any three-manifold $\HF^+(Y)$ is of $\HF^+$-type. There is a natural notion of similar skylines for modules of $\HF^+$-type, defined exactly as in the case of finitely generated $\F[U]$-modules.

\begin{lem}\label{lem:heightchanges} Let $n\in\Z_{\geq 0}$. Suppose that we have an exact triangle 
\begin{equation}
\centering
\begin{tikzcd}  \F^n \ar[dr, "f"] && B \ar[ll,"g"]  \\&A\ar[ur,"h"]  \end{tikzcd}           
 \end{equation}
of either finitely generated $\F[U]$-modules, or modules of $\HF^+$-type. Then $A$ and $B$ have similar skylines.
\end{lem}
\begin{proof}

We treat the case of finitely generated $\F[U]$-modules. The case of modules of $\HF^+$ type can be proved analogously.

Let $x$ be an element of $\F^n$. Observe that $Uf(x)=0$.  Write $\F^n\simeq\F^{n_a}\oplus\F^{n_b}$ and $A= \underset{1\leq j\leq m_a}{\bigoplus}T_{h_a(j)}^-$, where we allow $h_a(j)=\infty$, with $f$ the map that sends the first $n_a$ components of $\F^n$ into the bottom generators of the first $n_a$ towers of $\underset{1\leq j\leq n_a}{\bigoplus}T_{h_a(j)}^-$.

Consider now the maps $h$ restricted to each tower $T^-_{h_a(j)}$. Since $h$ is $U$ equivariant we have that $U^{h(a_j)-1}h(T_{h(a_j)}^-)=0$ for $1\leq j\leq n_a$ and that $U^{h(a_j)}h(T_{h(a_j)}^-)=0$ for $ n_a<j\leq m_a$. Since $h$ is injective when restricted to elements that are not at the bottom of the first $m_a$ towers we can write $B= \underset{1\leq j\leq n_a}{\bigoplus}T_{h_b(j)}^-\oplus \underset{n_a<j\leq m_a}{\bigoplus}T_{h_b(j)}^-\oplus  \underset{m_a< j\leq m_b}{\bigoplus}T_{h_b(j)}^-$, where $h_b(j)\geq h_a(j)-1$ for $1\leq j\leq n_a$ and $h_b(j)\geq h_a(j)$ for $n_a< j\leq m_a$. Here we again allow $h_a(j)=\infty$. Note that $g$ is $U$ equivariant and injective on the complement of the image of $h$. It follows in turn that  $h_a(j)\geq h_b(j)\geq h_a(j)-1$ for $1\leq j\leq n_a$ and $h_a(j)+1\geq h_b(j)\geq h_a(j)$ for $n_a< j\leq m_a$ and $h_a(j)=1$ for $m_a< j\leq m_b$. The result follows.
\end{proof}

\subsection{Large integer surgeries}
In this section we prove that large integer surgeries on knots have Heegaard Floer homology that satisfies the strong geography restriction.

As is customary we label $\spin^c$ structures on the surgery three-manifold $S^3_n(K)$ with integers $-|n|/2< s \leq |n|/2$ so that we can use the Ozsv\' ath-Szab\' o  surgery formula to get an identification 
\[H_*(A^+_s) \simeq \HF^+(S^3_n(K), s) \ .\]
for any integer $n \geq 2g(K)-1$, where $g(K)$ denotes the Seifert genus of the knot $K$. 

\begin{lem}\label{lem:largesimilar}
Let $K \subset S^3$ be a knot, and $n\geq 2g(K)-1$.  Then $\HF^+(S^3_n(K),s)$ and $\HF^+(S^3_n(K),s+1)$ have similar skylines.
\end{lem}
\begin{proof}
It follows from the surgery formula that for $n\geq 2g(K)-1$, $\HF^+(S^3_n(K),m)$ is isomorphic to $H_*(A_m)$ for any $m\in\Z\cap(-|n|/2,|n|/2]$. Note that there is a short exact sequence of chain complexes
\begin{center}
 \begin{tikzcd}
 0 \ar[r] & C(j=m,i<0) \ar[r] & A_m \ar[r] & A_{m+1} \ar[r] & 0
\end{tikzcd}
\end{center}
 which yields an exact triangle
\begin{center}
\begin{tikzcd} H_*(j=m,i<0) \ar[dr] && H_*(A_{m+1}) \ar[ll]  \\&H_*(A_m) \ar[ur] \end{tikzcd}           
\end{center}
That is, we have the following exact triangle
\begin{equation} \begin{tikzcd}  H_*(j=m,i<0) \ar[dr] && \HF^+(S^3_n(K),m+1) \ar[ll]  \\&\HF^+(S^3_n(K),m) \ar[ur] \end{tikzcd}          
\end{equation}
Observe that $ H_*(j=m,i<0)\simeq\F^k$ for some $k\in\Z_{\geq 0}$ so we can apply Lemma~\ref{lem:heightchanges} and obtain the desired result. 
\end{proof}

\begin{cor}\label{cor:largestrong}Let $K \subset S^3$ be a knot. For any integer $n\geq 2g(K)-1$ the module 
\[\HF^+(S^3_n(K))\simeq \bigoplus_{-|n|/2< i \leq |n|/2} H_*(A^+_s) \]
satisfies the strong geography restriction. 
\end{cor}
\begin{proof} If $K$ is an $L$-space knot there is nothing to prove, otherwise there is a $\spin^c$ structure labelled by some integer $s$ (with $-|n|/2 < s \leq |n|/2$) containing a copy of $\F[U]/U^N$, where $N=\min\{ k \geq 0 \ | \ U^k\cdot \HF^+(S^3_n(K)) =0\}$, as direct summand. Applying Lemma \ref{lem:largesimilar} repeatedly we conclude that
\[ \HF^+(S^3_n(K),s), \  \HF^+(S^3_n(K),s+1),\  \HF^+(S^3_n(K),s+2), \  \dots    \]
form a sequence where each term has a similar skyline to its successor. Furthermore, since $A_s\simeq\CF^+(S^3)$ for $s\geq g(K)$, we have that the modules in the sequence are eventually isomorphic to $\HF^+(S^3)\simeq T^+$. At this point we observe that the intermediate value principle holds, that is: given a sequence of modules with similar skylines that is eventually torsion-free, if the first module contains a copy of $T^+_N=\F[U]/U^N$ then a copy of $T^+_i$ for every $1\leq i<N$ should appear at some point along the sequence. It follows that 
\[T_N^+\oplus  T_{N-1}^+\oplus \dots \oplus T_1^+\subseteq \bigoplus_{ s\leq i \leq |n|/2 } \HF^+(S^3_n(K),i) \ ,\]
and we are done. 
\end{proof}

\subsection{Small integer surgeries} In this section we prove that the Lin geography restriction holds for integer surgeries on non $L$-space knots. This can be viewed as a warm up for the next section in which we show that the Heegaard Floer homology of any non-trivial surgeries on any knot satisfies the strong geography restriction. Recall that a knot $K \subset S^3$ is an $L$-space knot if $S^3_{n}(K)$ is an $L$-space for some integer $n \geq 0$. 

 \begin{lem}\label{lem:largestrong}Let $K \subset S^3$ be a knot. If $L$ is not an $L$-space knot then the module $\HF^+(S^3_n(K))$ satisfies the Lin geography restriction. 
\end{lem}
\begin{proof}As explained in Section \ref{sec:surgery} the surgery formula expresses the Heegaard Floer chain complex of the surgery $\CF^+(S^3_{n},s)$ as  the mapping cone $\mathbb{X}^+(n) $ of a map $D_n:\mathbb{A}^+ \to \mathbb{B}^+$. This has an associated short exact sequence 
\begin{center}
 \begin{tikzcd}
 0 \ar[r] & \mathbb{B}^+ \ar[r] & \mathbb{X}^+(n) \ar[r] & \mathbb{A}^+ \ar[r] & 0
\end{tikzcd}
\end{center} 
inducing an exact triangle 
\begin{center}
\begin{tikzcd}
\HF^+(S^3_n(K))\arrow[dr,"q_*"]&&H_*(\mathbb{B}^+)\arrow[ll,"i_*"]\\&H_*(\mathbb{A}^+)\arrow[ur,"(D_n)_*"]
\end{tikzcd}
\end{center}
whose connecting homomorphism is identified with the map $(D_n)_*:H_*(\mathbb{A}^+) \to H_*(\mathbb{B}^+)$ induced in homology. 

It is not hard to see that the map $(D_n)_*$ is surjective. Indeed given $b\in B^+_j \subset \mathbb{B}$ one can find a sequence $(a_s)_{s\in \Z} \in \mathbb{A}^+$ such that: $v_j(a_j)=b$, $v_{j+kn}(a_{j+kn})=h_{j+(k-1)n}(a_{j+(k-1)n})$, and $a_s=0$ otherwise. Consequently, we have a short exact sequence:
\begin{center}
 \begin{tikzcd}
 0 \ar[r] & \HF^+(S^3_n(K))   \ar[r] & H_*(\mathbb{A}^+) \arrow[r,"(D_n)_*"] & H_*(\mathbb{B}^+) \ar[r] & 0
\end{tikzcd}
\end{center} 
In particular $\HF^+(S^3_n(K))$ is identified with the kernel of $(D_n)_*$. Since $(D_n)_*$ sends the reduced part of $H_*(\mathbb{A}^+)$ to zero,  the result now follows from the proof of Corollary~\ref{cor:largestrong} where it was shown that $H_*(\mathbb{A}^+)$ contains an $\F$-summand. 
\end{proof}

 \subsection{Extension to the case of rational surgeries} We shall now extend the results from the previous sections, and prove Theorem~\ref{knots}. 
 
Suppose that $K$ is a knot in $S^3$. Let $V_i$ and $H_i$ be the invariants considered in \cite{ni2010cosmetic}, and $M_i:=\min\{V_i,H_i\}:\Z\to\Z_{\geq 0}$. Since they will be used frequently in this section we recall some properties of the invariants $M_i$.

\begin{lem}\label{lem:eventuallyzerotower}
    For sufficiently large $|i|$, $M_i=0$. \end{lem}
\begin{proof}
  This follows from the fact that $V_i$ and $H_{-i}$ are zero for sufficiently large $i$.
\end{proof}

\begin{lem}\label{lem:eventuallyinfinitytower}
    $\lim_{i\to\infty}H_i=\lim_{i\to-\infty}V_i=\infty$. \end{lem}
\begin{proof}
  This follows from the definition of $V_i$ and $H_i$.
\end{proof}

\begin{lem}\label{lem:symmetry}
    $M_{-i}=M_{i}$. 
\end{lem}
\begin{proof}
This follows from the definition of $M_i$ together with the fact that $V_i=H_{-i}$.
\end{proof}

\begin{lem}\label{lem:changingsizetower}
  For $i\geq 0$ $M_i\geq M_{i+1}\geq M_{i}-1\geq 0$ and $M_{-i}\geq M_{-i-1}\geq M_{-i}-1\geq 0$. 
\end{lem}

\begin{proof}
	Observe that after restricting to the infinite height towers in $H_*(A_s)$ and $H_*(A_{s+1})$ we have that $(h_s^+)_*=(h_{s+1}^+)_*\circ j:$ where $j$ is either $0$ or multiplication by $U$. It follows that $H_i\leq H_{i+1}\leq H_i+1$. Similarly $V_i\leq V_{i-1}\leq V_i+1$. Since $M_0=V_0=H_0$ the result follows.
\end{proof}

Note that combining Lemma~\ref{lem:changingsizetower} and Lemma~\ref{lem:eventuallyzerotower} we see that $M_i$ is bounded from above. Indeed, Lemma~\ref{lem:symmetry} and the proof of Lemma~\ref{lem:changingsizetower} show that $M_i$ is bounded from above by $M_0$.

We now make some statements that already appeared in \cite[Section 3.3]{ni2014characterizing}. The rational surgery formula expresses the module $\HF^+(S^3_{p/q},s)$ as the mapping cone of the map ${{D_s}:\mathbb{A}_s^+\to \mathbb{B}_s^+}$ we described at the end of Section \ref{sec:surgery}. It follows that there is an exact triangle:
\begin{center}
\begin{tikzcd}
\HF^+(S^3_{p/q},s)\arrow[dr,"q_*"]&&H_*(\mathbb{B}_s^+)\arrow[ll,"i_*"]\\&H_*(\mathbb{A}_s^+)\arrow[ur,"(D_s)_*"]
\end{tikzcd}
\end{center}
where the connecting homomorphism can be identified with  $(D_s)_*$, the map induced on homology by $D_s$, and $i_*$ and $q_*$ are the maps induced by the relevant inclusion and quotient chain  maps.

\begin{lem}\label{lem:plusker}
For a knot $K\subset S^3$ we have that $\HF^+(S^3_{p/q}(K),s)\simeq \ker((D_s)_*)$.
\end{lem}
\begin{proof}
It follows from the construction of $D_s$ that $(D_s)_*$ is surjective. Thus $i_*$ is trivial and $\HF^+(S^3_{p/q},s)\simeq \ker((D_s)_*)$.
\end{proof}

 Set 
\[H_{\text{red},*}(\mathbb{A}_s^+):=\underset{i\in\Z}{\bigoplus} (H_{\text{red},*}(A^+_{\left\lfloor \frac{s+pi}{q}\right\rfloor}))\]
and let $\ker((D_s)_{*,T})$ denote the kernel of the restriction of $(D_s)_{*}$ to the infinite height towers of $H_*(\mathbb{A}^+_s)$.

As in the case of the proof of Lemma~\ref{lem:largestrong}, note that for any $s,p,i,q$, an element $x\in H_*(A_{\lfloor\frac{s+ip}{q}\rfloor})$ can be extended canonically to an element of $\tilde{x}\in\ker((D_s)_*)$, where the components of $\tilde{x}$ are in the towers of $H_*(A_{\lfloor\frac{s+jp}{q}\rfloor})$. This relies on the fact that for any $n\in\Z^{\geq 0}$ there is a sufficiently large positive $s$ $h_s^+$ kills the bottom $n$ generators of $A_s^+$ and $v_{-s}^+$ kills the bottom $n$ generators of $A^+_{-s}$ -- see Lemma~\ref{lem:eventuallyinfinitytower}, for instance.

\begin{lem}\label{lem:decomposingsmall}
$\HF^+(S^3_{p/q},s)\simeq H_{\text{red},*}(\mathbb{A}_{s}^+)\oplus \ker((D_s)_{*,T})$.
\end{lem}

\begin{proof}
The module $H_*(\mathbb{A}_s^+)$ can be expressed as the direct sum of $H_{\text{red},*}(\mathbb{A}_{s}^+)$ and a ``free" part consisting only of $T^+$-summands, one for each $A^+_{\left\lfloor \frac{s+pi}{q}\right\rfloor}$-summand. Furthermore, the  map $(D_s)_*:H_*(\mathbb{A}_s^+) \to H_*(\mathbb{B}_s^+)$ maps $H_{\text{red},*}(\mathbb{A}_{s}^+)$ to zero. 
\end{proof}

For each $s$, pick an $i_s\in\Z$ with $M_{i_s}=\underset{i\in\Z}{\max}\{M_{\left\lfloor\frac{s+ip}{q}\right\rfloor}\}$.
\begin{cor}\label{cor:reddecomp}
    $\HF^+_{\text{red}}(S^3_{p/q},s) \simeq H_{\text{red},*}(\mathbb{A}_{s}^+)\oplus \underset{i\neq i_s\in\Z}{\bigoplus}T^+_{M_{\left\lfloor \frac{s+ip}{q}\right\rfloor}}$.
\end{cor}

\begin{proof} Recall from Lemma~\ref{lem:eventuallyzerotower} that $M_i=0$ for all but finitely many $i\in\Z$.
    Given the previous lemma, it suffices to determine the classes which are in the infinite height tower in $\ker((D_s)_{*,T})$ . Observe that in an arbitrary Maslov grading there are only finitely many generators of $H_*(\mathbb{A}_s^+)$. For any sufficiently large Maslov grading $a$ in which $H_*(\mathbb{A}_s^+)$, the sum of generators of Maslov grading $a$ is in $\ker((D_s)_*)$. Let $x_a$ denote the sum of generators of Maslov grading $a$. It follows that $x_a$ is in the tower of $\ker((D_s)_*)$ for sufficiently large $a$. It follows that the bottom of the tower coincides with one of the towers $T_{M_{\left\lfloor\frac{s+i_0p}{q}\right\rfloor}}$ where $i_0$ is picked so that $M_{\left\lfloor\frac{s+i_0p}{q}\right\rfloor}=\underset{i\in\Z}{\max}\left\{M_{\left\lfloor\frac{s+ip}{q}\right\rfloor}\right\}$ or equivalently -- given Lemma~\ref{lem:changingsizetower} -- any $i_0$ which minimizes $\left|\left\lfloor\frac{s+i_0p}{q}\right\rfloor\right|$.
\end{proof}

We now show that the two summands of $$\HF^+_{\text{red}}(S^3_{p/q})\simeq\left( \underset{0\leq s< p}{\bigoplus} H_{\text{red},*}(\mathbb{A}_{s}^+)\right)\oplus \left(\underset{0\leq s<p}{\bigoplus}\left(\underset{i\neq i_s\in\Z}{\bigoplus}T^+_{M_{\left\lfloor \frac{s+ip}{q}\right\rfloor}}\right)\right)$$ in Corollary~\ref{cor:reddecomp} satisfy the strong geography restriction. We begin with a preparatory Lemma to treat the first summand:

\begin{lem}\label{lem:AnsummandtoHFsummands} Suppose that $H_*(A_n)$ contains $m$ $T^+_k$ summands. Then $H_{\text{red},*}(\mathbb{A}_{s}^+)$ contains at least $m\left|\left\{i:\left\lfloor\dfrac{s+pi}{q}=n\right\rfloor\right\}\right|$ many summands of the form $T^+_k$ .
\end{lem}
\begin{proof}

 It follows from Lemma~\ref{lem:decomposingsmall} that if $H_{\text{red},*}(\mathbb{A}^+_s)$ contains $m$ $T^+_k$ summand then for each $i$ with $\left\lfloor\dfrac{s+pi}{q}\right\rfloor=n$ we can find $m$ distinct $T^+_k$ summands in $H_*(\mathbb{A}_s^+)$. The result follows.
\end{proof}

Note that  $\left|\left\{i:\left\lfloor\dfrac{s+pi}{q}=n\right\rfloor\right\}\right|\neq 0$ for some $s$ and consequently that Lemma \ref{lem:AnsummandtoHFsummands} always imply the existence of at least one $T^+_k$ summand.

\begin{cor}\label{cor:torsionsummand}
   $ \underset{0\leq s< p}{\bigoplus}H_{\text{red},*}(\mathbb{A}_{s}^+)$ satisfies the strong geography restriction.
\end{cor}
\begin{proof}
    Consider the largest finite height tower in $ H_{\text{red},*}(\mathbb{A}_{s}^+)$,
 $T_k^+$. $T_k^+$ is supported in $ H_{\text{red},*}(A_{s}^+)$ for some $s$. Observe that for sufficiently large $s'$, $ H_{\text{red},*}(A_{s'}^+)$ has no finite height towers. Observe too that $ H_{\text{red},*}(A_{t}^+)$ and $ H_{\text{red},*}(A_{t+1}^+)$ have similar skylines for all $t$, as in the proof of Lemma~\ref{lem:largesimilar}. It follows that for all $1\leq j\leq k$ there exists some $t$ such that $H_{\text{red},*}(A_t)$ contains a $T_j^+$ summand. The result now follows by an application of Lemma~\ref{lem:AnsummandtoHFsummands}, noting that for all $t$, $H_{\text{red},*}(A_t)$ occurs as a summand of $\underset{0\leq s< p}{\bigoplus}H_{\text{red},*}(\mathbb{A}_{s}^+)$.
\end{proof}

We now study $\underset{0\leq s< p}{\bigoplus}\left(\underset{i\neq i_s\in \Z}{\bigoplus}T^+_{M_{\left\lfloor\frac{s+ip}{q}\right\rfloor}}\right)$.  The key property we use in the proof is that removing a subset of tallest finite height towers from an $\F[U]$-module which satisfies the strong geography restriction (if they exist) yields a new $\F[U]$-module which still satisfies the strong geography restriction.
\begin{lem}\label{lem:towersummand}
      $\underset{0\leq s< p}{\bigoplus}\left(\underset{i\neq i_s\in \Z}{\bigoplus}T^+_{M_{\left\lfloor\frac{s+ip}{q}\right\rfloor}}\right)$ satisfies the strong geography restriction.
\end{lem}
\begin{proof}
Observe that $s$ takes every value in the range $0\leq s\leq p-1$. It follows that $\frac{s+ip}{q}$ takes the value of every fraction with denominator $q$ as $0\leq s\leq p-1$ and $i\in \Z$. Thus $\underset{0\leq s< p}{\bigoplus}\left(\underset{i\in \Z}{\bigoplus}T^+_{M_{\left\lfloor\frac{s+ip}{q}\right\rfloor}}\right)$ satisfies the strong geography restriction, since for $k\geq 0$ $M_K\geq M_{k+1}\geq M_{K}-1$   and $\lim_{|k|\to\infty}M_k=0$ -- see Lemma~\ref{lem:eventuallyzerotower} and Lemma~\ref{lem:changingsizetower}. However, we are considering the submodule where the $i$ values are subject to the additional restraint that $i\neq i_s$, so we require some  additional analysis.

Fix $s\in\{0,1,\dots p-1\}$. Without loss of generality we may take $\frac{p}{q}>0$ -- since $S^3_{\frac{p}{q}}(K)\simeq S^3_{-\frac{p}{q}}(-K)$. Observe that for each $s$ there is a unique $i$ for which \begin{align*}
       -\dfrac{p}{2q}\leq \dfrac{s+ip}{q}< \dfrac{p}{2q}
   \end{align*}

 It follows that there exists at least one $i$ for which 
 \begin{align}
       -\left\lfloor\dfrac{p}{2q}\right\rfloor-1\leq \left\lfloor\dfrac{s+ip}{q}\right\rfloor\leq  \left\lfloor\dfrac{p}{2q}\right\rfloor
   \label{eq:1}\end{align}
Observe then that $T_{M_n}^+$ appears in $\underset{0\leq s< p}{\bigoplus}\left(\underset{i\neq i_s\in \Z}{\bigoplus}T^+_{M_{\left\lfloor\frac{s+ip}{q}\right\rfloor}}\right)$ for all $n\geq \left\lfloor\dfrac{p}{2q}\right\rfloor$ and  $n\leq -\left\lfloor\dfrac{p}{2q}\right\rfloor-1$. Indeed, if for all $s$ there is a unique $i$ which satisfies Equation~\ref{eq:1}, then that $i$ is $i_s$ and 
$\underset{0\leq s< p}{\bigoplus}\left(\underset{i\neq i_s\in \Z}{\bigoplus}T^+_{M_{\left\lfloor\frac{s+ip}{q}\right\rfloor}}\right)$ satisfies the strong geography restriction.

We now consider the case in which there exists more than one $i$ value satisfying Equation~\ref{eq:1}. Without loss of generality we can take $i_0$  to be the smallest solution to Equation~\ref{eq:1}. Observe that the set of solutions to Equation~\ref{eq:1} are then given by $\{i_0+k:k\in\Z,0\leq k\leq K\}$ for some fixed $K>0$. We seek to determine $i_0$. Observe that if $i_0$ satisfies Equation~\ref{eq:1} and $\left\lfloor\dfrac{s+i_0p}{q}\right\rfloor>-\left\lfloor\dfrac{p}{2q}\right\rfloor$ then for all $0<k\in\Z$ we have that

\begin{align*}
\left\lfloor\dfrac{s+(i_0+k)p}{q}\right\rfloor&\geq \left\lfloor\dfrac{s+i_0p}{q}\right\rfloor+\left\lfloor\dfrac{kp}{q}\right\rfloor\\&>-\left\lfloor\dfrac{p}{2q}\right\rfloor+\left\lfloor\dfrac{kp}{q}\right\rfloor\\&\geq \left\lfloor\dfrac{p}{2q}\right\rfloor+\left\lfloor\dfrac{(k-1)p}{q}\right\rfloor\\&\geq\left\lfloor\dfrac{p}{2q}\right\rfloor \end{align*}

Contradicting the assumption that $i_0$ is not a unique solution to Equation~\ref{eq:1}. It follows that $\left\lfloor\dfrac{s+i_0p}{q}\right\rfloor\in\left\{-\left\lfloor\dfrac{p}{2q}\right\rfloor,  -1-\left\lfloor\dfrac{p}{2q}\right\rfloor \right\}$. Indeed, we see that if $\left\lfloor\dfrac{s+i_0p}{q}\right\rfloor=  -1-\left\lfloor\dfrac{p}{2q}\right\rfloor$ then $\left\lfloor\dfrac{s+(i_0+k)p}{q}\right\rfloor=  \left\lfloor\dfrac{p}{2q}\right\rfloor-1$ for all $0<k\leq K$, contradicting the definition of $i_0$. It follows that $\left\lfloor\dfrac{s+i_0p}{q}\right\rfloor=-\left\lfloor\dfrac{p}{2q}\right\rfloor$ and $\left\lfloor\dfrac{s+(i_0+k)p}{q}\right\rfloor=\left\lfloor\dfrac{p}{2q}\right\rfloor$ for all $0<k\leq K$.

Recalling then that $T_{M_n}^+$ appears in $\underset{0\leq s< p}{\bigoplus}\left(\underset{i\neq i_s\in \Z}{\bigoplus}T^+_{M_{\left\lfloor\frac{s+ip}{q}\right\rfloor}}\right)$ for all $n\geq \left\lfloor\dfrac{p}{2q}\right\rfloor$ we see that removing $\underset{0\leq s<p}{\bigoplus}T^+_{M_{\left\lfloor\frac{s+(i_0+k)p}{q}\right\rfloor}}$ from  $\underset{0\leq s< p}{\bigoplus}\left(\underset{i\in \Z}{\bigoplus}T^+_{M_{\left\lfloor\frac{s+ip}{q}\right\rfloor}}\right)$ removes some collection of tallest height towers, each of height $M_{\lfloor p/q\rfloor}$. Thus even after removing these towers we obtain an $\F[U]$-module which satisfies the strong geography restriction, as desired.
\end{proof}

We can now prove Theorem~\ref{knots}.

\begin{proof}[Proof of Theorem \ref{knots}]
Suppose $Y$ is $\frac{p}{q}\in\Q$ surgery on a knot $K$ in $S^3$. By Lemma~\ref{lem:decomposingsmall} $\HF^+(S^3_{p/q}(K))$ consists of two summands. The two summands satisfy the strong geography restriction by Lemma~\ref{lem:towersummand} and Corollary~\ref{cor:torsionsummand}. The result follows.
\end{proof}

\section{Three-manifolds arising as surgery on links}\label{sec:links}
A link $L=L_1\cup \dots\cup L_\ell\subset S^3$ with $\ell$ components can be represented by a multi-pointed Heegaard diagram $\mathcal{H}=(\Sigma, \alpha_1\dots \alpha_k , \beta_1, \dots , \beta_k, z_1, \dots z_\ell, w_1,\dots w_\ell )$. Similarly to the case of knots, this gives rise to a multi-filtered chain complex $(\CFL^-(\mathcal{H}), \mathcal{F}_1, \dots. \mathcal{F}_\ell)$ where: $\CFL^-(\mathcal{H})$ denotes the free $\F[U_1, \dots U_\ell]$-module generated by the intersection points of the lagrangian tori $\mathbb{T}_\alpha$ and $\mathbb{T}_\beta$ associated to the $\alpha$- and $\beta$-curves, equipped with differential
\[\partial \x= \sum_{\y	\in \mathbb{T}_\alpha \cap \mathbb{T}_\beta }\ \  \sum_{\phi\in \pi_2(\x,\y), \ \mu(\phi)=1}  \#\left( \frac{\mathcal{M}(\phi)}{\R}\right)  U_1^{n_{z_1(\phi)}} \dots \ \ U_\ell^{n_{z_\ell(\phi)}} \cdot \y \ , \]
and $\mathcal{F}_{i}=\mathcal{F}_{L_i}$ denotes the filtration associated to the $i^\text{th}$ component of the link $L$ extended so that multiplication by the variable $U_i$ drops the filtration level $\mathcal{F}_{i}$ by one~\cite{ozsvath2008holomorphic}. 

\subsection{Alexander filtration and $\spin^c$ labelling}One can view that the multi-filtration $\vec{\mathcal{F}}= (\mathcal{F}_{1} \dots , \mathcal{F}_{\ell})$ as taking value in the lattice:
\[ \mathbb{H}(L)= \bigoplus_{i=1}^\ell \mathbb{H}(L)_i \ , \ \ \ \ \ \  \mathbb{H}(L)_i =\Z+\frac{1}{2} \sum_{j\not= i} \text{lk}(L_i, L_j)\ . \]
whose points parametrize relative $\spin^c$ structures on $S^3\setminus L$. Lattice points can also be interpreted as $\spin^c$ structures on any integer surgery $S^3_{n_1, \dots , n_\ell}(L)$,  with the rule that two lattice points $\mathbf{s}$ and $\mathbf{s}' \in \Hy(L)$ define the same  $\spin^c$ structure if their difference $\mathbf{s} - \mathbf{s}'$ lies in the span of the columns of the linking matrix $\Lambda$:
\[\Lambda_{ij}= 
\begin{cases} 
& lk(L_i,L_j) \ \ \ \text{ if } i\not=j  \\
& \ \ \ \ n_i \ \ \ \ \ \ \ \ \text{ otherwise }
\end{cases} 	 \ . \]
To avoid ambiguity one can restrict to a fundamental domain $P(\Lambda)\subset \Hy(L)$ of the action so that every point $\mathbf{s}\in P(\Lambda)$ corresponds to a unique $\spin^c$ structure of   $S^3_{n_1, \dots , n_\ell}(L)$.

Sometimes instead of working with the fundamental domain $P(\Lambda)$ we shall consider a  small enlargement of it that we shall denote by $\overline{P(\Lambda)}$. This consists of all the points in $P(\Lambda)$ union all points $\bs \in \mathbb{H}(L)$ whose coordinates differ from a point in $P(\Lambda)$ by at most one. This can be interpreted as ``the lattice closure" of $P(\Lambda)$.

\begin{exa} Consider  $n$-surgery on a knot $K \subset S^3$.  In the previous section we chose $P(\Lambda)=(-n/2, n/2] \cap \Z$. In this case $\overline{P(\Lambda)}= (-n/2-1, n/2+1] \cap \Z $.  
\end{exa}

\begin{exa} For an algebraically split link $L=L_1\cup \dots L_\ell$ in $S^3$, that is a link such that $\text{lk}(L_i,L_j)=0$ for all $i\not=j$, the linking matrix is diagoinal 
\[\Lambda= 
\left( \begin{matrix}
n_1 & & \\
&\ddots & \\
& & n_\ell 
\end{matrix} \right)\]
and one can pick $P(\Lambda)$  so that it has the shape of a hyperbox with vertices at the points with coordinates $(\pm n_1/2, \dots , \pm n_\ell/2)$.
\end{exa}

\begin{exa} 
The linking matrix of a framed $2$-component link $L\subset S^3$ is a $2\times2$ matrix 
\[\Lambda= 
\left( \begin{matrix}
n_1 & \lambda \\
\lambda & n_2 
\end{matrix} \right)\]
where $\lambda= \text{lk}(L)$. In this case $P(\Lambda)$ has the shape of a parallelogram with vertices at the points with coordinates
\[
\pm \frac{1}{2} 
\left( \begin{matrix}
n_1  \\
\lambda 
\end{matrix} \right) 
\pm \frac{1}{2}
\left( \begin{matrix}
\lambda \\
n_2 
\end{matrix} \right)  \ .
\] 
Its two sides have length $(n_1^2+ \lambda^2)^{1/2}$ and $(n_2^2+ \lambda^2)^{1/2}$ respectively, and acute angle $\theta$ such that:
\[\cos \theta = \left(\frac{1}{n_1} + \frac{1}{n_2} \right) \cdot \frac{\lambda}{ \left(\left(\frac{\lambda}{n_1}\right)^2+ 1\right)^{1/2} \cdot  \ \ \left(\left(\frac{\lambda}{n_2}\right)^2+ 1\right)^{1/2} } \ . \]
Note that as the framings $n_1$ and $n_2$ grow larger $P(\Lambda)$ approaches a rectangle ($\theta \to \pi/2$ as $n_1, n_1 \to \infty$) and its sides grow in length.
\end{exa}

For a general $\ell$-component link the domain $P(\Lambda)$ has the shape of a parallelopiped with sides parallel to the column vectors of the linking matrix $\Lambda$. Note that as the framings grow larger $P(\Lambda)$ approaches the shape of an $\ell$-dimensional hyperbox. 
We shall take $P(\Lambda)$ to be the unique domain with the points
\[\pm \frac{1}{2} \sum_{i=1}^\ell \vec{\Lambda}_i\]
as corners.

\subsection{The Ozsv\' ath-Manolescu surgery formula} Given $\mathbf{s} = (s_1, \dots , s_\ell) \in \mathbb{H}(L)$, as in the case of knots, we define $A_\mathbf{s}^-=\mathfrak{A}^-(\mathcal{H}, \mathbf{s})$ to be the subcomplex of $\CFL^-(\mathcal{H})$ spanned by all generators $x=U_1^{m_1}\dots U_\ell^{m_\ell}\cdot \mathbf{x}$ with $\x \in \mathbb{T}_\alpha \cap \mathbb{T}_\beta$, and $m_1, \dots , m_\ell \geq 0$ such that $ \mathcal{F}_i(\x)-m_i= \mathcal{F}_i(x) \leq s_i$ for $i = 1,\dots , \ell$.

\begin{thm}[Large surgery formula, Ozsv\' ath \& Manolescu \cite{manolescu2010heegaard}]\label{OMformula} Let $\mathbf{n}=(n_1, \dots, n_\ell)\in \Z^\ell$, and $\mathbf{s} \in P(\Lambda)$ a $\spin^c$ structure. If $\mathbf{n}$  is sufficiently large, that is $n_i>>0$ for $i\in\{1, \dots, \ell\}$, then $\CF^-(S^3_\mathbf{r}(L), \mathbf{s})$ and $A^-_\mathbf{s}$ are quasi-isomorphic. 
\end{thm}

Note that multiplication by any of the variables $U_1, \dots, U_\ell$ gives rise to a chain map $A^-_\mathbf{s}\to A^-_\mathbf{s}$. It turns out that these maps are chain homotopic to each other, and that in Theorem \ref{OMformula} the map induced on homology can be identified with the $U$-action of the Heegaard Floer group $\HF^-(S^3_\mathbf{n}(L), \mathbf{s})$.

\subsection{A definition of large} Given a link $L$ we let $\mathcal{L}$ denote its link Floer polytope. This  is the convex hull of all those $\bs \in  \mathbb{H}(L)$ for which $\widehat{\HFL}(L, \bs)\not=0$. In this context we will denote by $Q=[-q_1,q_1]\times[-q_2,q_2]\dots [-q_n,q_\ell]$ the smallest hyperbox containing $\mathcal{L}$. Note that the link Floer polytope can be expressed  as the Minkowski sum of the dual of the Thurston polytope and of a hyperbox of side length one, unless $L$ contains a split unlinked component~\cite[Theorem 1.1]{ozsvath2008linkFloerThurstonnorm}. In particular if $L$ does not have a split unknotted component then $Q$ has non-vanishing volume. 

\begin{lem}\label{lem:large}
If $\mathbf{n}>>0$ then $Q\subset P(\Lambda)$.
\end{lem}
\begin{proof}
For a given $\mathbf{n}$ the hyperbox $Q$ will be covered by a few different translates of the fundamental domain $P(\Lambda)$. Let $\Lambda_t$ be the matrix obtained from the linking matrix $\Lambda$ by increasing its diagonal enteries by $t>0$. As $t\to +\infty$ the $\ell$-dimensional parallelogram  $P(\Lambda_t)$ increase in size and eventually contains all the translates of $P(\Lambda)$ covering the hyperbox $Q$. 
\end{proof}

\begin{defi}\label{def:large}
$\mathbf{n}$-framed surgery on a link $L$ is large if $Q\subset \overline{P(\Lambda)}$.
\end{defi}

Note that if we take the framing $\mathbf{n}$ to be large in the sense of Definition \ref{def:large} then we can apply the large surgery formula stated in Theorem \ref{OMformula}.  This is explained implicitly by Manolesu-Ozsv\'ath~\cite[P 112.]{manolescu2010heegaard}, see also the proof given by Borodzik-Gorsky for~\cite[Lemma 3.10]{borodzik2018immersed}. To expand on this point slightly, observe that Manolescu-Ozsv\'ath's general integer surgery formula can be ``horizontally truncated" -- see~\cite[Section 10.1]{manolescu2010heegaard} -- to the large surgery formula in the case that $\Lambda$ is large in the sense of Definition~\ref{def:large}. The key observation that allows one to prove this is that appropriate versions of the map from~\cite[Lemma 10.1]{manolescu2010heegaard} are isomorphisms for $\mathbf{s}\not\in Q$.

\subsection{Lemma \ref{lem:largesimilar} revisited} We now move towards proving Theorem \ref{largelinks}. We start by introducing some notation, definition, and lemmas.

\begin{defi}Let $\bs=(s_1, \dots, s_\ell)$ and $\bs'=(s_1', \dots, s_\ell')\in \mathbb{H}(L)$ be two lattice points. We say that $\bs'$ is adjacent to $\bs$ if  $0\leq s_i'-s_i\leq 1$ for $i=1, \dots, \ell$. 
\end{defi}

The key lemma is as follows.

\begin{lem} \label{adj} Let $\bs'$ and $\bs$ be two lattice points. If $\bs'$ is adjacent to $\bs$  then $H_*(A^-_{\bs'})$  and $H_*(A^-_{\bs})$ have similar skylines.
\end{lem}
\begin{proof}
If $\bs'$ is adjacent to $\bs$ then there is an inclusion of sub-complexes  $A^-_{\bs}\subseteq A^-_{\bs'}$. Furthermore $H_*(A^-_{\bs'} / A^-_{\bs})$ is annihilated by multiplication by $U$. Thus we can look at the short exact sequence: 
 \[\begin{tikzcd}
 0 \ar[r] & A^-_{\bs}\ar[r] & A^-_{\bs'} \ar[r] & A^-_{\bs'} / A^-_{\bs} \ar[r] & 0
\end{tikzcd}\]
and obtain an exact triangle of the form:
\begin{equation} \begin{tikzcd}  \F^k \simeq H_*(A^-_{\bs'} / A^-_{\bs}) \ar[dr] && H_*(A^-_{\bs}) \ar[ll]  \\&H_*(A^-_{\bs'}) \ar[ur] \end{tikzcd}\label{exacttrinaglelarge}           
\end{equation}         
For some $k\geq 0$. It follows from Lemma \ref{lem:changingsizetower} that  $H_*(A^-_\bs)$ and $H_*(A^-_{\bs'})$ have similar skylines. 
\end{proof}

We conclude this section with a discussion of integer homology spheres that arise as large surgeries on links.

\begin{lem}\label{prop:polytope}
Suppose large surgery on an $\ell$-component link $L$ yields an integral homology sphere. Then the link Floer polytope of $L$ is contained in $[-1,1]^\ell$.  
\end{lem}
\begin{proof}
    Suppose $L$ and $\ell$ are as in the statement of the proposition and let $S^3_\mathbf{n}(L)$ be a large integer homology sphere surgery on $L$. Observe that the only lattice point that $P(\Lambda)$ contains is $\mathbf{0}$. It follows that $\overline{P(\Lambda)}=[-1,1]^\ell$. By the definition of large we have that $\mathcal{L}\subseteq Q\subseteq [-1,1]^\ell$, as desired.
\end{proof}

\begin{proof}[Proof of Proposition~\ref{cor:Fsummands}]
    Suppose $Y$ is as in the statement of the corollary. Write $Y=S^3_\mathbf{n}(L)$ where $\mathbf{n}$ is large and $L$ has $\ell$ components. By Lemma~\ref{prop:polytope} we have that $L$ has link Floer polytope supported in $[-1,1]^\ell$. It follows that $H_*(A^-_{(1,1,\dots,1)})\simeq T^-$. Since $Y$ is an integral homology sphere $P(\Lambda)$ contains only the origin $\mathbf{0}\in\mathbb{H}(L)$. Observe that $\mathbf{0}$ is adjacent to $(1,1,\dots,1)$ so the result follows from Lemma~\ref{adj}. 
\end{proof}

\subsection{Lattice paths}We shall consider paths of adjacent lattice points.

\begin{defi}
A sequence of lattice points $(\bs_0, \bs_1, \dots , \bs_k)$  is a lattice path if each lattice point $\bs_{j+1}$ in the sequence is adjacent to his predecessor $\bs_{j}$.
\end{defi}

We now prove a preliminary version of Theorem \ref{largelinks}.
\begin{thm}\label{linlinks} Suppose that $L\subset S^3$ is a link with $\ell$ components. Then there exists some $\mathbf{n}_0 \in \Z^\ell_{\geq 0}$ such that for all  $\mathbf{n}\geq \mathbf{n}_0$ the surgery three-manifold $S^3_\mathbf{n}(L)$ satisfies the Lin geography restriction. 
\end{thm}
\begin{proof}[Proof of Theorem \ref{linlinks}]
Let $Q=[-q_1,q_1]\times[-q_2,q_2]\dots [-q_n,q_\ell]$ be a hyperbox containing $\mathcal{L}\subset \mathbb{H}(L)$, the link Floer polytope of $L$. It is shown in \cite{HolomorphicdiskslinkinvariantsandthemultivariableAlexanderpolynomial} that $\CFL^-(L)$ is generated  by module elements $x\in \CFL^-(L)$ that have Alexander multi-grading $\mathbf{s}\in \mathcal{L}$. Indeed,  $\vec{\mathcal{F}}(\x)\in \mathcal{L}$ for all intersection points $\x \in \mathbb{T}_\alpha \cap \mathbb{T}_\beta$. Consequently $H_*(A^-_\mathbf{s})\simeq \HF^-(S^3)=\F[U]$ for all lattice points $\bs=(s_1, \dots, s_\ell)$ such that 
\[\mathcal{L}\subset \{(j_1, \dots, j_\ell) \in \mathbb{H}(L) :  s_i \leq j_i \}\ .\] 

Choose a lattice point $\bs \in P(\Lambda)$ such that  $H_*(A_{\bs}^-)$ contains a direct summand $\F[U]/U^n$. If there is no such $\bs$ then $L$ is an $L$-space link and there is nothing to prove. Starting from $\bs$ form a lattice path $(\bs_0=\bs, \bs_1, \dots , \bs_k)$ where  each vector is obtained from its predecessor by increasing all components by one. Taking $k\geq1$ to be sufficiently large we can ensure that $\bs_k$ lies in the semispace $\{(j_1, \dots, j_\ell) \in \mathbb{H}(L) :  q_i \geq j_i \}$ so that $\mathcal{L}\subset A^-_\bs$, and $H_*(A^-_\mathbf{s})=\F[U]$.

Of course if $\mathbf{n}>>0$ then the lattice points  in the path $(\bs_0, \bs_1, \dots , \bs_k)$ all belong to $P(\Lambda)$ and by  Theorem \ref{OMformula} we have identifications 
\[T^-_n\subset \HF^-(S^3_{\mathbf{n}}(L), \bs_0)\simeq H_*(A^-_{\bs_1})\ ,  \ \dots \ ,\   \HF^-(S^3_{\mathbf{n}}(L), \bs_k)\simeq H_*(A^-_{\bs_k})=\F[U] \ . \]
At this point we can apply Lemma \ref{adj} to conclude that consecutive terms in this sequence have similar skylines. Thus a term in the sequence has an $\F$-summand and in turn we have that $\HF^-(S^3_{\mathbf{n}}(L))$ satisfies the Lin geography restriction. 
\end{proof}

\subsection{Proof of Theorem~\ref{largelinks}} We continue the discussion initiated in the previous section, maintaining the same notation, with a simple Lemma.

\begin{lem}\label{lem:deputy}
Suppose that $\bs \in \mathbb{H}(L)$ is some lattice point. Then there is a lattice point $\bs' \in Q$ such that $H_*(A^-_{\bs}) \simeq H_*(A^-_{\bs'})$. 
\end{lem}
\begin{proof}
 Suppose that $L$ is an $\ell$ component link. Let $Q=[-q_1,q_1]\times[-q_2,q_2]\dots [-q_\ell,q_\ell]$ and $\mathbf{s}=(s_1,s_2,\dots s_\ell)$. Suppose that $i$ is an index such that $s_i \in \R \setminus [-q_i,q_i]$.  It suffices to show that $H_*(A_{\mathbf{s}}^-)\simeq H_*(A_{\mathbf{s}_i}^-) $ where $\mathbf{s}_i=(s_1,\dots s_{i-1},q_i\dots s_\ell)$ if $s_i$ is positive or $\mathbf{s}_i=(s_1,\dots s_{i-1},-q_i\dots s_\ell)$ if $s_i$ is negative.
	
\emph{Case one: $s_i>0$.} We can represent the generators of the link Floer chain complex $\CFL^-$ as a collection of lattice points: we mark a lattice point $\mathbf{j}=(j_1, \dots, j_\ell)\in \mathbb{H}(L)$ as in $\CFL^-$ if there is an $\F$-generator $x=U_1^{m_1} \dots U_\ell^{m_\ell} \cdot \x$ such that  $j_i= \mathcal{F}_i(x)= \mathcal{F}_i(\x)-m_i$. In this picture multiplication by a variable $U_i$ acts as \[(j_1, \dots, j_\ell)\mapsto (j_1, \dots, j_i-1, \dots, j_\ell) \ , \] 
and the sub-complex $A_\bs^-$ is identified with the set of generators lying inside the semi-space of lattice points $\mathbf{j}=(j_1, \dots, j_\ell)\in \mathbb{H}(L)$ with coordinates  $j_i\leq s_i$, for $i=1, \dots, \ell$. 

Note that $\CFL^-$ can be completely reconstructed by translating the lattice points corresponding to the $\F[U_1, \dots, U_\ell]$-generators (intersection points in the Lagrangian Floer set-up), and that these all lie inside the Floer polytope $\mathcal{L}$ \cite{ozsvath2008holomorphic}. It follows that there are no lattice points marked as generators in the strip
$\{\mathbf{j} \leq \bs\}\setminus \{\mathbf{j} \leq \bs_i\} \subset \mathbb{H}(L)$, and we have the equality $A_{\mathbf{s}}^-=A_{\mathbf{s}_i}^-$.
	
\emph{Case two: $s_i<0$.} Observe that $\CFL^-(L)$ comes equipped with a $U_i$ filtration for all $1\leq i\leq n$. Let $\{K_i\}_{1\leq i\leq\ell}$ be the components of $L$. Observe that the effect of switching the orientation of the component $K_i$ is to switch the $U_i$ and $\mathcal{F}_i$ filtrations; see~\cite{ozsvath2008holomorphic}. We introduce some new notation to keep track of this; if $L$ is a link let $A_{(s_1,s_2,,\dots s_n),t}^-(L)$ denote the subcomplex of $\CFL^-(L)$ generated by elements with each $\mathcal{F}_j$ filtration at most $s_j$ and $U_i$ filtration at most $t$. Now let $L_i$ denote the link $L$ with the orientation of $L$ reversed. It follows that $$H_*(A^-_{\mathbf{s}}(L))\simeq H_*(A^-_{(s_1,s_2,\dots s_{i-1}, 0,s_{i+1}\dots s_\ell),s_i}(L_i)).$$

Since $s_i< -q_i$, multiplication by $U_i^{s_i}$ induces an isomorphism:

$$ H_*(A^-_{(s_1,s_2,\dots s_{i-1}, -s_i,s_{i+1}\dots s_n),0}(L_i))\to H_*(A^-_{(s_1,s_2,\dots s_{i-1}, 0,s_{i+1}\dots s_n),s_i}(L_i))$$
As in case one we have that $H_*(A^-_{(s_1,s_2,\dots -s_i,\dots s_n ),0},L_i)\simeq H_*(A^-_{(s_1,s_2,\dots q_i,\dots s_n),0},L_i)$ for all $i$. The result follows. 
\end{proof}

\begin{remark}
Note that the argument of Lemma \ref{lem:deputy} specifies the constant $b$ in the statement of~\cite[Lemma 10.1]{manolescu2010heegaard}.
\end{remark}

\begin{proof}[Proof of Theorem \ref{largelinks}] 
Suppose that  $Y\simeq S^3_\mathbf{n}(L)$ for some $\ell$ component link $L\subset S^3$ with $\mathbf{n}$ large in the sense of Definition \ref{def:large}. Set $$N:= \min \{k\geq0 : U^k \cdot \HF^-_\text{red}(Y)=0\}.$$ If $N=0$ or $N=1$ there is nothing to prove,  thus we can assume that $N>1$.

Given a lattice path   $( \bs_1, \dots , \bs_k)$ consisting of lattice points in   $\overline{P(\Lambda)}$, if we can ensure that the sequence satisfies the following conditions:
\begin{enumerate}
\item $\bs_i \in P(\Lambda)$ for $i\in \{1, \dots, k-1\}$, so that for each of these lattice points we can apply Theorem~\ref{OMformula} to show that $H_*(A^-_\bs)\simeq \HF^-(S^3_{\mathbf{n}}(L), \bs_1)$,
\item the start-point $\mathbf{s}_1$ is such that $\HF^-(S^3_{\mathbf{n}}(L), \bs_1)$ contains a copy of $T^-_N$,
\item the end-point $\mathbf{s}_k$ is such that $ H_*(A^-_{\bs_k})$ is torsion free.
\end{enumerate} 
then using Lemma \ref{adj} we can conclude that the summand $\bigoplus_{i=1}^{k-1} \HF^-(S^3_{\mathbf{n}}(L), \bs_i)$ contains a copy of  $ T^-_N\oplus T^-_{N-1}\dots \oplus T_1^- $.

Choose a lattice point $\bs\in P(\Lambda)$ representing a $\spin^c$ structure such that $\HF^-(Y, \bs) \simeq \HF^-(S^3_{\mathbf{n}}(L), \bs)$ contains a $T^-_N$-summand. Let $Q=[-q_1,q_1]\times[-q_2,q_2]\dots [-q_\ell,q_\ell]$ be the smallest hyperbox containing the link Floer polytope, and $\mathbf{q}=(q_1, \dots, q_\ell)$. By Lemma~\ref{lem:deputy} there is a lattice  point $\bs'$ in $Q\subset \overline{P(\Lambda)}$ that  also contains a $T^-_N$-summand. If $\bs'\in  P(\Lambda)$ then we can find a path $( \bs_1=s',\bs_2 \dots , \bs_k=\mathbf{q})$, with $\bs_1 \dots , \bs_{k-1} \in Q\cap P(\Lambda)$, and we are done.

Of course it could be that $\bs'\in Q\cap \partial\overline{P(\Lambda)}$, where $\partial \overline{P(\Lambda)}=\overline{P(\Lambda)} \setminus P(\Lambda) $. In this limiting case we can find an adjacent point $\bs'' \in Q \setminus \partial\overline{P(\Lambda)}=Q\cap P(\Lambda)$ and start a lattice path from $\bs''$ instead. This shows that $\HF^-(S^3_{\mathbf{n}}(L))$ contains a copy of $ T^-_{N-1}\dots \oplus T_1^- $ as well as a copy of $T^-_N$, concluding the proof.
\end{proof}

\bibliographystyle{plain}
\bibliography{bibliography}
\end{document}